\documentclass{article}
\usepackage{amsthm,amsmath}
\usepackage[colorlinks,citecolor=blue,urlcolor=blue]{hyperref}
\usepackage{amsmath,amssymb}
\usepackage{amsfonts}
\usepackage{verbatim}
\usepackage{longtable}
\usepackage{lscape}
\usepackage{graphicx}
\usepackage{inputenc}
\usepackage{hyperref}
\usepackage{color}

\numberwithin{equation}{section}

\usepackage[english]{babel}

\usepackage[T1]{fontenc}
\usepackage{amscd}

\usepackage{color}

\theoremstyle{plain}
\newtheorem{theorem}{Theorem}[section]
\newtheorem{thm}{Theorem}[section]

\newtheorem{proposition}[theorem]{Proposition}

\theoremstyle{definition}
\newtheorem{definition}[theorem]{Definition}
\newtheorem{defn}[theorem]{Definition}
\newtheorem{assu}[theorem]{Assumption}
\newtheorem{prop}[theorem]{Proposition}
\newtheorem{rem}[theorem]{Remark}
\newtheorem{cor}[theorem]{Corollary}

\DeclareMathAlphabet{\mathbbo}{U}{bbold}{m}{n}

\newcommand{\bbR}{\mathbb{R}}

\newcommand{\cDc}{\mathcal{D}}

\newcommand{\mcP}{\mathcal{P}}

\newcommand{\Rn}{\bbR^d}
\newcommand{\bRb}{\bbR}

\newcommand{\rmd}{\mathrm{d}}

\newcommand{\E}{\mathbb{E}}
\newcommand{\EE}{\mathbb{E}}
\newcommand{\PP}{\mathbb{P}}
\newcommand{\R}{\mathbb{R}}

\title{Self-interacting diffusions: long-time behaviour and exit-problem in the convex case}

\author{
A.~Aleksian$^{1,a}$, P.~Del Moral$^{2,b}$, A.~Kurtzmann$^{3,c}$\\
and J.~Tugaut$^{1,d}$\\[5pt]
\small{$^1$Universit\'e Jean Monnet, CNRS UMR 5208,}\\
\small{Institut Camille Jordan, Maison de l'Universit\'e, 10 rue Tr\'efilerie,}\\
\small{CS 82301, 42023 Saint-\'Etienne Cedex 2, France}\\
\small{$^2$ INRIA Bordeaux Research Center, France.}\\
\small{$^3$ Universit\'e de Lorraine, Institut Elie Cartan de Lorraine, CNRS UMR 7502,}\\
\small{Vandoeuvre-l\`es-Nancy, F-54506, France.}\\[5pt]
\small{$^a$ ashot.aleksian@univ-st-etienne.fr, $^b$ pierre.del-moral@inria.fr}\\
\small{$^c$ aline.kurtzmann@univ-lorraine.fr, $^d$ julian.tugaut@univ-st-etienne.fr}
}

\begin{document}

\maketitle

\begin{abstract}

We study a class of time-inhomogeneous diffusion: the self-interacting one. We show a convergence result with a rate of convergence that does not depend on the diffusion coefficient. Finally, we establish a so-called Kramers' type law for the first exit-time of the process from domain of attractions when the landscapes are uniformly convex.



\end{abstract}

\section{Introduction}
\label{s:intro}

%


In this work, we are interested in a time-inhomogeneous diffusion. More precisely, we study the following specific diffusion, driven by the Stochastic Differential Equation (SDE):
\begin{equation}
\label{eq:sid}
\mathrm{d}X_t = \sigma \rmd B_t - \left( \nabla V (X_t) + \frac{1}{t}\int_0^t \nabla W( X_t -  X_s) \rmd s \right) \rmd t, \quad X_0 = x\in \Rn
\end{equation}
where $V,W$ are two potentials on $\bRb^d$ and $\sigma>0$. The precise assumptions on the potentials will be given later in Subsection~\ref{ss:assu}.
{\color{black} 
We can already notice that the current position of the process $X_t$ depends on the whole past trajectory of the process $(X_s)_{0\le s\le t}$ through the interaction potential $W$ appearing in the drift term. We call this kind of process a path-interaction process.}

\subsection{History}
\label{ss:history}

Path-interaction processes have been introduced by Norris, Rogers and Williams during the late 80s in~\cite{NRW}. Since this period, they have been an intensive research area. Under the name of Brownian Polymers, Durrett and Rogers~\cite{DR} studied a family of self-interacting diffusions, as a model for the shape of a growing polymer.  Denoting by $X_t$ the location of the end of the growing polymer at time $t$, the process $X$ satisfies a SDE driven by a Brownian motion,  with a drift term depending on its own occupation measure. One is then interested in finding the scale for which the process converges to a non trivial limit. Later, another model of growing polymer has been introduced by Bena\"im, Ledoux and Raimond~\cite{BLR}, for which the drift term depends on its own empirical measure. Namely, they have studied  the following process living in a compact
smooth connected Riemannian manifold~$M$ without boundary:
\begin{equation*}
\mathrm{d}X_t = \sum_{i=1}^N F_i(X_t)\circ\mathrm{d}B_t^i - \int_M \nabla_x W(X_t,y)\mu_t(\mathrm{d}y)\mathrm{d}t,
\end{equation*}
where $W$ is a (smooth) interaction potential, $(B^1,\cdots,B^N)$
is a standard Brownian motion on $\mathbb{R}^N$, $\mu_t:=\frac{1}{t} \int_0^t \delta_{X_s} \rmd s$ and the symbol
$\circ$ stands for the Stratonovich stochastic integration. 
In the compact setting, they have shown that the asymptotic behaviour of the empirical measure of the process can be related to the analysis of some deterministic dynamical flow. Later, Bena\"im and Raimond~\cite{BR} gave sufficient conditions for the almost sure convergence of the empirical measure (again in the compact setting). More recently, Raimond~\cite{Ra} has generalized the previous study and has proved that for the solution of the SDE living on a compact manifold $$\rmd X_t = \rmd B_t - \frac{g(t)}{t} \int_0^t \nabla_x V(X_t, X_s) \rmd s \ \rmd t$$ unless $g$ is constant, the approximation of the empirical measure by a deterministic flow is no longer valid.

Similar questions have also been answered in the non-compact setting, that is~$\Rn$. Chambeu and Kurtzmann~\cite{CK} have studied the ergodic behaviour of the self-interacting diffusion depending on the empirical mean of the process. They have proved, under some convexity assumptions (ensuring the non-explosion in finite time of the process), a convergence criterion for the diffusion solution to the SDE $$\rmd X_t = \rmd B_t - g(t) \nabla V\left(X_t - \frac{1}{t}\int_0^t X_s \rmd s \right)  \ \rmd t$$
where $g$ is a positive function. 
This model could represent for instance the behaviour of some social insects, as ants who are marking their paths with the trails' pheromones. This paper shows in particular how difficult is the study of general self-interacting diffusions in non-compact spaces as in~\cite{AK}, driven by the generic equation $$\rmd X_t = \rmd B_t - \frac{1}{t} \int_0^t \nabla_x V(X_t, X_s) \rmd s \ \rmd t.$$ Nevertheless, if the interaction function $V$ is symmetric and  uniformly convex, then Kleptsyn and Kurtzmann~\cite{kk-ejp} obtained the limit-quotient ergodic theorem for the self-attracting diffusion. Moreover, they managed to obtain the rate of convergence.

\subsection{Close processes}
\label{ss:closedmodel}

Another problem related to this paper is the diffusion corresponding to McKean-Vlasov partial differential equation. This corresponds to the measure-dependent drift diffusion governed by the SDE
\begin{equation}
\rmd X_t = \rmd B_t - \nabla W * \nu_t (X_t)\rmd t
\end{equation}
where $\nu_t:= \mathcal{L}(X_t)$, $W$ is a smooth convex potential and $\ast$ stands for the convolution. The asymptotic behaviour of $X$ has been studied by various authors these last years, see for instance Cattiaux, Guillin and Malrieu~\cite{CGM}. It turns out that under some assumptions, the law $\nu_t$ converges, as $t$ goes to infinity, to the (unique if $W$ is strictly convex) probability measure solution to the equation
\[
\nu:=\frac{1}{Z} \exp\left(-2W\ast\nu\right)\,,
\]
where $Z:=Z(\nu)$ is the normalization constant. In the latter paper, the authors use a particles system to prove both a convergence result (with convergence rate) and a deviation inequality for solutions of granular media equation when the interaction potential is strictly convex. To this end, they use a uniform propagation of chaos property and a control in Wasserstein distance of solutions starting from different initial conditions. 

A related question to this problem concerns the first exit-times from domains of attraction for the following motion
\begin{equation}\label{eq:MKV}
\rmd X_t^\sigma = \sigma \rmd B_t - \nabla V(X_t^\sigma) \rmd t - \nabla W\ast\nu_t^\sigma(X_t^\sigma) \rmd t\,,
\end{equation}
where $V$ is a potential, $\nu_t^\sigma:= \mathcal{L}(X_t^\sigma)$ and $\sigma >0$. This was addressed by Herrmann, Imkeller and Peithmann~\cite{HIP}, who exhibited a Kramers' type law for the diffusion exit from the potential's domains of attraction by a large deviations principle for the self-stabilizing diffusion (which is a peculiar instance of McKean-Vlasov diffusion). To get this, they reconstructed the Freidlin-Wentzell theory for the self-stabilizing
diffusion. More precisely, they established a large deviations principle with a good rate function. The exit-problem for the McKean-Vlasov diffusion has also been already studied recently, with a different method. More precisely, in~\cite{T2011f}, Tugaut analyses the exit-problem (time and location) in convex landscapes, showing the same result as Herrmann, Imkeller and Peithmann, but without reconstructing the proofs of Freidlin and Wentzell. Then, he generalized very recently his results in the case of double-wells landscape in~\cite{JOTP}. In~\cite{Tug14c,Alea}, he does not use large deviations principle but a coupling method
between the time-homogeneous diffusion $$\rmd Y_t^\sigma = \sigma \rmd B_t - \nabla V (Y_t^\sigma) \rmd t - \nabla W (Y_t^\sigma - m) \rmd t$$ (where $m$ is the unique point at which the vector field $\nabla V$ equals 0) and the McKean-Vlasov diffusion so that the results on the exit-time of $Y^\sigma$ can be used for the exit-time of the self-stabilizing diffusion~\eqref{eq:MKV}.

\subsection{The questions that we study here}
\label{ss:problems}

A large family of path-dependent processes has been studied by Saporito, see for instance~\cite{yuri}. He proves, with his co-authors, existence and uniqueness of such processes. The difference with the process studied here is that we normalize the occupation measure. In the current paper, we also prove the existence and uniqueness result for general potentials $V$ and $W$, which are not necessarily convex.

A second result that we are obtaining is related to the asymptotic behaviour. Indeed, after proving the existence and the uniqueness of the solution to Diffusion~\eqref{eq:sid}, we are studying the convergence in long-time of the probability measure $\mu_t$. The idea is similar to the one in \cite{kk-ejp}. 

The present paper also deals with the exit-time problem of~\eqref{eq:sid}. We prove that the first exit-time $\tau$ of the diffusion from some domain satisfies a Kramers' type law in the following sense:
\[
\lim_{\sigma\to0}\frac{\sigma^2}{2}\log(\tau)=H>0\,,
\]
where the convergence holds in probability.

We could adapt the techniques introduced by Herrmann, Imkeller and Peithmann but only in the case of a convex potential $V$. Our aim is to generalize the study also to non-convex potentials. In the present work, we will solve the exit-problem (time and location) for Diffusion~\eqref{eq:sid}. As will be shown later, the exit-location can be easily obtained once we know the asymptotics of the first exit-time.

\subsection{Outline}
\label{ss:outline}

Our paper is divided into three parts. First, Section~\ref{s:results} is devoted to the explanation of the precise assumptions and the statement of the main results. 
After that, in Section~\ref{s:exitproblem}, we prove Theorem~\ref{dell} and Corollary~\ref{hphp}, that is we establish the Kramers' type law and the exit-location result. To do so, we first provide some intermediate results. Finally, some possible extensions  are discussed in the Appendix. 
Before finishing the introduction, we give the notations used in the paper.

\subsection{Notations}
\label{ss:notations}

The parameters $\kappa$, $\xi$, $\epsilon$ and $\delta$ are arbitrarily small. The constants are denoted by $C$ as usual and generic.

As usual, we denote by $\mathcal{M}(\Rn)$ the space of signed
(bounded) Borel measures on $\Rn$ and by $\mathcal{P}(\Rn)$ its
subspace of probability measures. 

 In the sequel, $(\cdot,\cdot)$ stands for the Euclidean
scalar product and $|\cdot|$ is the associated norm. 

We first introduce the notion of positively invariant domain.

\begin{definition}
\label{stasta}
Let $d$ be any positive integer. Let $\mathcal{G}$ be a subset of $\mathbb{R}^{d}$ and let $U: \mathbb{R}^{d} \to \mathbb{R}^d$ be a vector field  satisfying some ``good assumptions''. For all $x\in\mathbb{R}^d$, we consider the dynamical system $\rho_t(x)=x+\int_0^tU\left(\rho_s(x)\right) \rmd{s}$. We say that the domain $\mathcal{G}$ is positively invariant for the flow generated by $U$ if the orbit $\left\{\rho_t(x)\,;\,t\in\mathbb{R}_+\right\}$ is included in $\mathcal{G}$ for all $x\in\mathcal{G}$.
\end{definition}



Let us recall the definition of the Wasserstein distance.
\begin{definition} 
  For $\mu_1,\mu_2 \in \mathcal{P}(\Rn)$, the quadratic Wasserstein distance is defined as $$\mathbb{W}_2(\mu_1,\mu_2) := \left(\inf \{ \mathbb{E}(|Z_1-Z_2|^2)\}\right)^{1/2},$$ where the infimum is taken over all the random variables such that $\mathcal{L}(Z_1)= \mu_1$ and $\mathcal{L}(Z_2)= \mu_2$. This corresponds to the minimal $L^2$-distance taken over all the couplings between $\mu_1$ and $\mu_2$.
  
  Similarly, the Wasserstein distance $\mathbb{W}_{2k}$ is defined as 
  $$\mathbb{W}_{2k}(\mu_1,\mu_2) := \left(\inf \{ \mathbb{E}(|Z_1-Z_2|^{2k})\}\right)^{1/(2k)}.$$
\end{definition}  

In the following, for readability issue, we will omit the $\sigma-$exponent for the process $X$ as well as the occupation measure $\mu_t$. Nevertheless, the reader has to keep in mind that the process $(X_t,t\geq0)$ and $\mu_t$ do depend on $\sigma$.

\begin{defn}
\label{def:m}
The minimizer of $V$ is denoted as $m$.
\end{defn}

We also introduce the following mapping on the probability measures:
\begin{equation*}
\Pi_\sigma(\mu)(dx):=\frac{e^{-\frac{2}{\sigma^2}\left(V(x)+W\ast\mu(x)\right)}}{\int_{\bRb^n}e^{-\frac{2}{\sigma^2}\left(V(y)+W\ast\mu(y)\right)dy}}\,dx\,.
\end{equation*}

\section{Assumptions and main results}
\label{s:results}

\subsection{Assumptions}
\label{ss:assu}

In this section we introduce the assumptions considered in the paper.

\begin{assu}
\label{assu:regu}
We assume some regularity for the potentials $V$ and $W$: $V \in \mathcal{C}^2(\mathbb{R}^d)$, $W\in\mathcal{C}^2(\mathbb{R}^d)$. Also, without loss of generality, we consider only potentials such that $V \geq 0, \; W \geq 0$.
\end{assu}
Assumption~\ref{assu:regu} is usual in SDE. Since we use It\^o calculus techniques, the differentiability assumption above is necessary.

\begin{assu}
\label{assu:growth1}
$V$ and $W$ (and their first two derivatives) have at most a polynomial growth. In other words, there exists a polynomial function~$P$ of degree~$2k$ such that $P(|x|)\geq1$ for any $x\in\bRb^n$, and
\begin{align}
\label{domination}
&|V(x)|+|W(x)|\leq P(|x|)\,,\\
\label{domination2}
&|\nabla V(x)|+|\nabla W(x)|\leq P(|x|)\,,\\
\label{domination3}
\mbox{and}\quad&\|\nabla^2 V(x)\|+\|\nabla^2 W(x)\|\leq P(|x|)\,.
\end{align}
\end{assu}

This assumption is used in the paper \cite{kk-ejp} to establish the rate of convergence towards the invariant probability measure. We come back to this question in Section \ref{ss:convrate}.

\begin{rem}
Note, that without any loss of generality, we can choose polynomial~$P$ to be such that $P(|x|) = C(1 + |x|^k)$. Then, the following property holds: there exists a constant $\gamma>0$ such that $P(|x + y|) \leq \gamma(P(|x|) + P(|y|))$.
\end{rem}

We also need the following assumption to establish the exit-time result:
\begin{assu}
\label{assu:growth2}
There exist $\rho,\alpha>0$ such that for any $x\in\bRb^n$, we have $\nabla^2V(x)\geq\rho{\rm Id}$ and $\nabla^2W(x)\geq\alpha{\rm Id}$. The unique minimizer of $V$ is denoted as $m$ and the unique minimizer of $W$ is $0$.
\end{assu}


To ensure the existence of the process, we will also use the following assumption, that will help us to exhibit a Lyapunov function:
\begin{assu}
\label{assu:existence}
$\underset{\vert x\vert\rightarrow \infty}{\lim} V(x) =+ \infty$,  $\underset{\vert x\vert\rightarrow \infty}{\lim} \frac{\vert\nabla V(x)\vert^2}{V(x)} = +\infty$ and there exists $a>0$ such that 
$\Delta V(x)\leq aV(x)$.
\end{assu}

\begin{rem}
By the latter growth condition, $\vert\nabla V\vert^2 - \Delta V$ is bounded by below.
\end{rem}

Finally, we will eventually assume the following, meaning that $W$ is rotationally invariant:
\begin{assu}
\label{assu:spherical}
There exists a function $G$ from $\bRb_+$ to $\bRb$ such that $W(x)=G(|x|)$.
\end{assu}

The assumptions on the domain from which the diffusion $X$ exits are the following.

\begin{assu}
\label{assu:domain}
The domain $\cDc$ is open and satisfies the following hypotheses:
\begin{enumerate}
 \item Let $\varphi$ be the solution to the following equation
\[
\varphi_t=x_0-\int_0^t\nabla V(\varphi_s)ds-\int_0^t\frac{1}{s}\int_0^s\nabla W(\varphi_s-\varphi_r)drds\,.
\]
Then, for any $t\in\bRb_+$, we have $\varphi_t\in\cDc$ and $\lim_{t\to\infty}\varphi_t=m\in\cDc$.
 \item The domain $\cDc$ is positively invariant for the flow generated by the vector field $x\mapsto-\nabla V(x)-\nabla W\ast\delta_m(x)$.
 \item For any $x\in\partial\cDc$, define the flow $\rho(x)$ as the solution to the equation
\[
\rho_t(x)=x-\int_0^t\nabla V(\rho_s(x))ds-\int_0^t\nabla W(\rho_s(x)-m)ds\,.
\]
Moreover, assume that the following limit holds: $$\lim_{t\to\infty}\rho_t(x)=m.$$
\end{enumerate}
\end{assu}

The assumption i) guarantees that, starting from  fixed point $x_0$, the deterministic process defined by \eqref{eq:sid} with $\sigma = 0$ converges towards the point of attraction $m$. Of course, we expect $X_t$ to follow this path with high probability for small enough $\sigma > 0$. In fact, as will be shown later, its empirical measure $\mu_t$ will also converge towards $\delta_m$ and, after some deterministic time, with high probability, will stay inside a defined-in-advance neighbourhood of $\delta_m$. We call this effect: stabilisation of the empirical measure. 

That leads to assumptions ii) and iii) above. After the ``stabilisation time'' we expect our drift term $V + W*\mu_t$ to have a similar effect as $V + W*\delta_m$. Thus, the last two assumptions guarantee that the process $X_t$ will forever tend to stay inside the domain $\mathcal{D}$ and be attracted towards the point $m$. This is a necessary assumption when considering the exit from a stable domain of attraction solely under the influence of small noise.

\subsection{Existence and uniqueness}
\label{ss:existence}



The first results that we will provide are about the existence and the uniqueness of the solution to the SDE~\eqref{eq:sid}. 

\begin{thm}
\label{thm:existence-W1}
Under the Assumptions~\ref{assu:regu} and~\ref{assu:existence}, for any $x_0 \in \mathbb{R}^n$, there exists a unique global strong solution $(X_t,t\geq 0)$ to Equation~\eqref{eq:sid}.
\end{thm}



\begin{proof}
Local existence and uniqueness of the solution to \eqref{eq:sid} is standard under the locally Lipschitz assumptions on the vector fields (see for instance~\cite[Theorem 13.1]{RW00}). 
We only need to prove here that $X$ 
does not explode in a finite time. Let us introduce the increasing sequence of stopping times $\tau_0 = 0$ and $$\tau_n :=
\inf\left\{t\geq \tau_{n-1}; \mathcal{E}_{t}(X_t) + \int_0^t \left\vert\nabla
\mathcal{E}_{s}(X_s)\right\vert^2 \mathrm{d}s > n\right\},$$ where $\mathcal{E}_{t}(X_t) := V(X_t) + \frac{1}{t} \int_0^t W(X_t-  X_s) \rmd s$. 
In order to show that the solution never explodes, we use the Lyapunov
functional $(x,t)\mapsto \mathcal{E}_t(x)$. As the process $(t,x)\mapsto\mathcal{E}_{t}(x)$ is of class $\mathcal{C}^2$ (in the space variable) and is a $\mathcal{C}^1$-semi-martingale
(in the time variable),
It\^o-Ventzell formula applied to
$(x,t)\mapsto\mathcal{E}_{t\wedge \tau_n}(x)$ implies
\begin{eqnarray}\label{eq:ito}
\mathcal{E}_{t\wedge \tau_n}(X_{t\wedge \tau_n}) = V(x_0)  + \int_0^{t\wedge \tau_n}
(\nabla \mathcal{E}_{s}(X_s), \mathrm{d}B_s) - \int_0^{t\wedge \tau_n} \left\vert\nabla \mathcal{E}_{s}(X_s)\right\vert^2 \mathrm{d}s\\
+ \frac{\sigma^2}{2}\int_0^{t\wedge \tau_n} \Delta \mathcal{E}_{s}(X_s)\mathrm{d}s
- \int_0^{t\wedge \tau_n}  \int_0^s  W(X_s-  X_u)\rmd u \frac{\mathrm{d}s}{s^2}.\notag
\end{eqnarray}
We note that $\int_0^{t\wedge \tau_n} (\nabla \mathcal{E}_{s}(X_s),
\mathrm{d}B_s)$ is a true martingale. 
By removing the negative terms and by using Assumption \ref{assu:existence}, we get the following bound on the expectation of the Lyapunov functional: 
\begin{equation*}
\mathbb{E}\mathcal{E}_{t\wedge \tau_n}(X_{t\wedge \tau_n})
\leq V(x_0) + a \int_0^{t}
\mathbb{E}\mathcal{E}_{s\wedge \tau_n}(X_{s\wedge \tau_n})
\mathrm{d}s.
\end{equation*}
So, Gronwall's Lemma leads to: $$\mathbb{E}V(X_{t\wedge \tau_n})\le \mathbb{E}\mathcal{E}_{t\wedge
\tau_n}(X_{t\wedge \tau_n}) \leq V(x_0)e^{at}.$$
Since, by Assumption \ref{assu:existence}, $\underset{\vert x\vert\rightarrow \infty}{\lim}V(x) =\infty$, the process $(X_t,t\geq 0)$ can not explode in a finite time, or else it leads to a contradiction with the inequality above. That proves existence of a unique global strong solution.
\end{proof}


\subsection{Convergence rate}
\label{ss:convrate}

In this section we show the long time behaviour for the empirical measure of the self-interacting diffusion in the convex landscape. This framework was considered in the paper of Kleptsyn and Kurtzmann~\cite{kk-ejp}. More precisely, they have proved the following

\begin{theorem}\cite[Theorem 1.6, Theorem 1.12, Proposition 2.5]{kk-ejp}
\label{t:main-2}
Let $X$ be the solution to the equation \eqref{eq:sid} with $\sigma=\sqrt{2}$. 
Suppose that $V$ and $W$ satisfy Assumptions~\ref{assu:regu}--\ref{assu:existence}. 
Then: 
\begin{enumerate}
    \item There exist $\alpha, C > 0$ such that for any $t \geq 0$ big enough: $\mu_t \in K_{\alpha, C}$ a.s., where
    \begin{eqnarray*}
        K_{\alpha,C} &:=& \{\mu \in \mcP(\Rn); \quad \forall R>0, \, \mu(\{y; |y|>R\})<Ce^{-\alpha R}\}.
    \end{eqnarray*}
    \item There exists a unique density $\rho_{\infty}:\Rn\to\bbR_+$, such that almost surely $$ \mu_t = \frac{1}{t}\int_0^t \delta_{X_s}\mathrm{d}s \xrightarrow[t\to+\infty]{*-weakly}
    \rho_{\infty}(x) \, \mathrm{d}x.
    $$
    \item There exists a constant $a>0$ such that almost surely, for $t$ large enough one has 
    \begin{equation*}
        \mathbb{W}_2(\mu_t, \rho_\infty) = O(\exp\{-a\sqrt[2k+1]{\log t}\})\,,
    \end{equation*}
    where $2k$ is the degree of the polynomial $P$ and $\mathbb{W}_2$ is the quadratic Wasserstein distance.
\end{enumerate}

\end{theorem}
Moreover, it was proved that, if $V$ is symmetric with respect to some point $q$, then the corresponding density $\rho_\infty$ is also symmetric with respect to the same point $q$. The authors showed that the density $\rho_{\infty}$ is the same limit density as in the result of~\cite{CMV}, uniquely defined by the following property: $\rho_{\infty}$ is a positive function, proportional to~$e^{-(V+W*\rho_{\infty})}$.

We note that the results of \cite{kk-ejp} were established for the case $\sigma = \sqrt{2}$. Nevertheless, one can check that each step of the proof can be reformulated with $\sigma$. Moreover, all the asymptotic (with respect to time $t$) results for small $\sigma$ can be upper bounded without loss of generality by the case of a constant $\sigma$, that one can take to be equal to $\sqrt{2}$. That means that for smaller $\sigma > 0$ we have faster convergence towards the invariant probability measure $\rho_\infty$.



We put together the observations above in the form of the following two results.






\begin{proposition}\label{prop:mu_t_K_alpha}
Under Assumptions~\ref{assu:regu}--\ref{assu:existence}, there exist $\alpha, C > 0$, such that for any $t \geq 0$ and for any $\sigma > 0$ small enough $\mu_t \in K_{\alpha, C}$ a.s. 
\end{proposition}

\begin{proposition}
\label{thm:KK12}
Under Assumptions~\ref{assu:regu}--\ref{assu:existence}, there exists a constant $a>0$ such that for any $x\in\bRb^n$, almost surely, the following asymptotics holds:
    \begin{equation*}
    \mathbb{W}_{2k}(\mu_t, \rho_\infty)=O\left(\exp\{-a\left(\log t\right)^{\frac{1}{2k+1}}\}\right)\,,
    \end{equation*}
where $\rho_\infty$ is the unique probability measure such that $\rho_\infty = \Pi_{\sigma}(\rho_\infty)$.
\end{proposition}

We stress that the convergence rate that we establish in Proposition~\ref{thm:KK12} does not depend on $\sigma$.

\subsection{Main result on exit-problem}
\label{ss:ep}

The main goal of this paper consists in finding some precise upper and lower bounds for the exit-time from some positively invariant domain.

\begin{theorem}
\label{dell}
We assume that the potentials $V$, $W$ and an open domain $\cDc$ satisfy the Assumptions~\ref{assu:regu}--\ref{assu:domain}. By $\tau:=\inf\left\{t\geq0\,\,:\,\,X_t\notin\cDc\right\}$, we denote the first time the process $X$ exits the domain $\mathcal{D}$. We introduce the so-called exit-cost:
\begin{equation}
\label{verdict}
H:=\inf_{x\in\partial\mathcal{D}}\left(V(x)+W(x-m)-V(m)\right)\,.
\end{equation}
Then $\displaystyle\mathbb{P}-\lim_{\sigma\to0}\frac{\sigma^2}{2}\log(\tau)=H$ that is for any $\delta>0$, we have
\begin{equation}
\label{eq:dell}
\lim_{\sigma\to0}\mathbb{P}\left(\exp\left\{\frac{2}{\sigma^2}\left(H-\delta\right)\right\}\leq\tau\leq\exp\left\{\frac{2}{\sigma^2}\left(H+\delta\right)\right\}\right)=1\,.
\end{equation}
\end{theorem}

This statement about the exit-time corresponds to what we denote as the Kramers' type law.

From Theorem \ref{dell}, we immediately obtain the classical statement on the exit-location.

\begin{cor}
\label{hphp}
Under the same assumptions as the ones of Theorem \ref{dell}, if $\mathcal{N}$ is a subset of $\partial\mathcal{D}$ such that $\displaystyle\inf_{z\in\mathcal{N}}\left(V(z)+W(z-m)-V(m)\right)>H$, then
\begin{equation}
\label{eq:hphp}
\lim_{\sigma\to0}\mathbb{P}\left(X_{\tau}\in\mathcal{N}\right)=0\,.
\end{equation}
\end{cor}

This means that the diffusion avoids to exit from a part of the boundary where the cost of exiting exceeds the exit-cost of $\cDc$.

\section{Exit-problem}
\label{s:exitproblem}

%
%

In this section, we prove our main result. First, we give the necessary intermediate results in Section~\ref{ss:intermediate}. More precisely, we show that there exists a time of stabilisation around $\delta_m$ for the occupation measure, in terms of Wasserstein distance. Then, we show that the process $X$ solution to \eqref{eq:sid} is close to the solution of the deterministic flow $(\varphi_t)_{t\ge 0}$. Using that, we prove in Corollary~\ref{cor:apple} that the probability of leaving a positively invariant domain before the occupation measure remains stuck in the ball of center $\delta_m$ and radius $\kappa$ for $\mathbb{W}_{2k}$ vanishes as $\sigma$ goes to zero. Then, we consider the coupling between the studied diffusion and the one where the occupation measure is frozen to $\delta_m$ and we show that these diffusions are close. 

After, we provide the proof of Theorem~\ref{dell} then we give the proofs of the intermediate results. Finally, we apply Theorem~\ref{dell} to level sets so that we are in position to prove the exit-location result in Section~\ref{ss:proofel}.

\subsection{Intermediate results}
\label{ss:intermediate}

We first introduce a deterministic time, representing the time of stabilisation of the occupation measure, if it occurs, around its supposed limit $\delta_m$: 

\begin{defn}
\label{def:def:Tkappa}
For any $\sigma>0$ and for any $\kappa>0$, we introduce:
\begin{equation}
\label{def:Tkappa}
T_\kappa(\sigma):=\inf\Big\{t_0\geq0\,\,:\,\,\forall t\geq t_0,\,\mathbb{E}\left(\mathbb{W}_{2k}\left(\mu_t; \delta_m\right)\right)\leq\kappa\Big\}\,,
\end{equation}
where we remind that $2k$ is the degree of the polynomial function $P$, introduced in Assumption~\ref{assu:growth1}.
\end{defn}

\begin{prop}
\label{prop:Tkappa:uniform}
For any $\sigma,\kappa>0$, the time $T_\kappa(\sigma)$ is finite. Moreover, for any~$\kappa>0$, there exists $T_\kappa>0$ such that
\begin{equation*}
\sup_{0<\sigma<1}T_\kappa(\sigma)\leq T_\kappa\,.
\end{equation*}
\end{prop}

The proof of Proposition~\ref{prop:Tkappa:uniform} is postponed to Section~\ref{proof:Tkappa:uniform}.

\medskip

Next, we show that the probability for the process $X$ to exit from $\cDc$ before the time $T_\kappa(\sigma)$ tends to $0$ as $\sigma$ goes to $0$.

We remind the reader that in this work, the noise vanishes. Consequently, it is natural to introduce the deterministic flow $(\varphi_t)_{t \geq 0}$ defined by the following zero-noise process
\begin{equation}
\dot{\varphi}_t = -\nabla V(\varphi_t) -\frac{1}{t}\int_0^t \nabla W(\varphi_t-\varphi_s) \rmd s, \quad \quad \varphi_0 = x_0.
\end{equation}

We will state that for any $T>0$, $(X_t,0\leq t\leq T)$ and $(\varphi_t,0\leq t\leq T)$ are uniformly close while the noise goes to zero. Namely,

\begin{proposition} 
\label{prop:apple}
We assume that the potentials $V$, $W$ and an open domain $\cDc$ satisfy the Assumptions~\ref{assu:regu}--\ref{assu:domain}. Then, for any~$\xi>0$ and for any~$T>0$, we have:
\begin{equation}
\label{apple}
\lim_{\sigma\to 0}\mathbb{P}\left(\sup_{t\in [0;T]}\left|X_t-\varphi_t(x_0)\right|^2>\xi\right)=0\,.
\end{equation}
\end{proposition}

The proof of Proposition~\ref{prop:apple} is postponed to Section~\ref{proof:apple}. We deduce immediately the following.

\begin{cor}
\label{cor:apple}
We assume that the potentials $V$, $W$ and an open domain $\cDc$ satisfy the Assumptions~\ref{assu:regu}--\ref{assu:domain}. Then:
\begin{equation}
\label{applekasso}
\lim_{\sigma\to 0}\mathbb{P}\left(\tau\leq T_\kappa(\sigma)\right)=0\,.
\end{equation}
\end{cor}

\begin{proof}
It is sufficient to consider $T:=T_\kappa$  and $\displaystyle\xi:=\inf_{t\in[0;T_\kappa]}{\rm d}\left(\cDc^c;\varphi_t\right)>0$ in Proposition~\ref{prop:apple}. We thus have:

\begin{align*}
\mathbb{P}\left(\tau\leq T_\kappa(\sigma)\right)&=\PP\left(\inf\Big\{{\rm d}\left(\cDc^c;X_t\right)\,\,:\,\,t\in[0;T_\kappa(\sigma)]\Big\}=0\right)\\
&\leq\PP\left(\sup_{t\in[0;T_\kappa(\sigma)]}|X_t-\varphi_t|>\xi\right)\,,
\end{align*}
which converges towards $0$ as noise vanishes.
\end{proof}

In \cite{Tug14c,JOTP}, Tugaut has proved the Kramers' type law for the exit-time. He has used a coupling between the diffusion of interest ($X$ here) and another diffusion that is expected to be close to $X$ if the time is sufficiently large. The main difficulty with the considered self-stabilizing diffusion is in fact that we do not have a uniform (with respect to time) control of the law.

Here, we have proved that the nonlinear quantity appearing in \eqref{eq:sid} (that is $\frac{1}{t}\int_0^t \delta_{X_s} \rmd s$) remains stuck - with high probability - in a small ball (for the~$\mathbb{W}_{2k}$-distance) of center $\delta_m$ and radius $\kappa$ for any $t\geq T_\kappa(\sigma)$. The idea is thus to substitute~$\frac{1}{t}\int_0^t \delta_{X_s} \rmd s$ by $\delta_m$ and to compare the new diffusion with the initial one.

We introduce the diffusion $(Y_t)_{t\ge 0}$ such that $Y_t=X_t$ if $t\leq T_\kappa(\sigma)$ and for any $t\geq T_\kappa(\sigma)$
\begin{equation}
\label{sandra}
{\rm d}Y_t=\sigma{\rm d}B_t-\nabla V\left(Y_t\right)\rmd t
-\nabla W(Y_t-m)\rmd t.
\end{equation}

\begin{proposition}
\label{prop:coupling}
We assume that the potentials $V$, $W$ and an open domain $\cDc$ satisfy the Assumptions~\ref{assu:regu}--\ref{assu:domain}. Then, if $\kappa$ is small enough, we have for any $\xi>0$ 
\begin{equation}
\label{eq:prop:coupling}
\limsup_{\sigma\to0}\mathbb{P}\left(\sup_{T_\kappa(\sigma)\leq t\leq\exp\left[\frac{2H+10}{\sigma^2}\right]}\left|X_t-Y_t\right|\geq\xi\right)\leq\kappa^k\,,
\end{equation}
where we remind that $2k$ is the degree of the polynomial function $P$ introduced in Assumption~\ref{assu:growth1}.
\end{proposition}
The proof of Proposition~\ref{prop:coupling} is postponed to Section~\ref{proof:coupling}.

\subsection{Idea of the proof concerning the exit-time}
\label{ss:proofet:idea}
The idea of the proof is to use the fact that diffusions $Y$ and $X$ are close to each other at least after the deterministic stabilisation time $T_{\kappa}(\sigma)$ and until some fixed deterministic time $\exp\left\{\frac{2(H + 5)}{\sigma^2}\right\}$. We choose this time to be sufficiently big for our line of reasoning. We can control the proximity of these two diffusions by parameter $\kappa$, which represents how close the empirical measure $\mu_t$ and $\delta_m$ are. It was already shown in Corollary~\ref{cor:apple}, that with $\sigma \to 0$ probability of exiting before time $T_{\kappa}(\sigma)$ tends to zero. That means that we can focus on the dynamics after the stabilisation of the occupation measure happens. For the upper bound, we show that the event $\left\{\tau >\exp\left[\frac{2(H + \delta)}{\sigma^2}\right]\right\}$ is unlikely due to the fact that, for small $\sigma > 0$, the diffusion $Y$ leaves a bigger (than $\mathcal{D}$) domain before the time $\exp\left[\frac{2(H + \delta)}{\sigma^2}\right]$, which, given the closeness of $X$ and $Y$, gives $\left\{\tau \leq \exp\left[\frac{2(H + \delta)}{\sigma^2}\right]\right\}$. Same type of reasoning is used to prove the lower bound $\exp\left\{\frac{2(H - \delta)}{\sigma^2}\right\}$. 

\medskip

Let us now provide the rigorous proof in the next section. The proof of the intermediate lemmas are given in Subsection~\ref{ss:prooflemmas}.

\subsection{Proof for the exit-time result}
\label{ss:proofet}

Fix some $\delta, \kappa > 0$, decrease it if necessary to be $\delta < 5$. For the upper bound, consider the following inequality:
\begin{equation}\label{eq:tau_lowbound1}
    \mathbb{P}(\tau > e^{\frac{2(H + \delta)}{\sigma^2}}) \leq \mathbb{P}(\tau > e^{\frac{2(H + \delta)}{\sigma^2}}, \tau^Y_{\mathcal{D}^e} \leq e^{\frac{2(H + \delta)}{\sigma^2}}) + \mathbb{P}(\tau^Y_{\mathcal{D}^e} > e^{\frac{2(H + \delta)}{\sigma^2}}),
\end{equation}
where $\mathcal{D}^e$ is some enlargement of domain $\mathcal{D}$ such that its exit-cost is equal to~$H + \frac{\delta}{2}$, i.e.: 
\begin{equation*}
    \mathcal{D}^e:= \{x \in \R^d: V(x) + W(x - m) - V(m) < H + \frac{\delta}{2} \};
\end{equation*}
and $\tau^Y_{\mathcal{D}^e}$ is the first exit-time of diffusion $Y$ from this domain, i.e.:
\begin{equation*}
    \tau^Y_{\mathcal{D}^e} := \inf \{t\,\,:\,\,Y_t \notin \mathcal{D}^e\}.
\end{equation*}
Note, that domain $\mathcal{D}^e$ (since both $V$ and $W$ are continuous and convex) satisfies the usual assumptions (see \cite{DZ}) and $d_e:= d(\mathcal{D},\partial\mathcal{D}^e) > 0$. By classical result of Freidlin-Wentzell theory,
\begin{equation*}
    \mathbb{P}\left(\tau^Y_{\mathcal{D}^e} > \exp\left[\frac{2((H + \delta/2)  + \delta/2)}{\sigma^2}\right]\right) \xrightarrow[\sigma \to 0]{} 0.
\end{equation*}
Let us decrease $\sigma_\kappa$ if necessary, such that the quantity above will be less then $\sqrt{\kappa}$ for any $\sigma < \sigma_\kappa$. Moreover, the first probability in \eqref{eq:tau_lowbound1} can be bounded by:
\begin{equation*}
    \mathbb{P}(\tau^Y_{\mathcal{D}^e} \leq e^{\frac{2(H + \delta)}{\sigma^2}} < \tau) \leq \mathbb{P}(|X_{\tau^Y_{\mathcal{D}^e}} - Y_{\tau^Y_{\mathcal{D}^e}}| \geq d_e) \leq 2\kappa^k,
\end{equation*}
where we use Proposition \ref{prop:coupling} and decrease $\kappa$ and $\sigma_\kappa$ if necessary.

We approach the lower bound similarly and introduce the contraction of the domain $\mathcal{D}$:
\begin{equation*}
    \mathcal{D}^c := \{x \in \bRb^d: V(x) + W(x - m) - V(m) < H - \frac{\delta}{2} \}\subset\cDc.
\end{equation*}
If $\mathcal{D}^c$ turns out to be empty, decrease $\delta$. As previously, the domain $\mathcal{D}^c$ satisfies usual properties and has positive distance with the boundary of the initial domain, that is~$d_c := d(\mathcal{D}^c,\partial\mathcal{D}) > 0$. We introduce the exit-time from the contracted domain for diffusion~$Y$:
\begin{equation*}
    \tau_{\mathcal{D}^c}^Y := \inf \{t\,\,:\,\,Y_t \notin \mathcal{D}^c\},
\end{equation*}
and have the following estimate:
\begin{equation*}
    \begin{aligned}
    \mathbb{P}( \tau < e^{\frac{2(H - \delta)}{\sigma^2}}) & \leq \mathbb{P}(T_{\kappa}(\sigma) < \tau < e^{\frac{2(H - \delta)}{\sigma^2}} \leq \tau_{\mathcal{D}^c}^Y) \\
    &\quad\quad + \mathbb{P}(\tau \leq T_{\kappa}(\sigma))  + \mathbb{P}(\tau_{\mathcal{D}^c}^Y \geq e^{\frac{2((H - \delta/2) - \delta/2)}{\sigma^2}}) \\
    & \leq \mathbb{P}(|X_{\tau} - Y_{\tau}| \geq d_c) ++ \mathbb{P}(\tau \leq T_{\kappa}(\sigma))+  2\kappa^k \\
    & \leq 3\kappa^k + \mathbb{P}(\tau \leq T_{\kappa}(\sigma))\\
		& \leq 4\kappa^k,
    \end{aligned}
\end{equation*}
by Corollary~\ref{cor:apple}, with $\kappa$ and $\sigma_\kappa$ small enough. This leads to:
\begin{equation*}
    \mathbb{P}(e^{\frac{2(H - \delta)}{\sigma^2}} \leq \tau \leq e^{\frac{2(H + \delta)}{\sigma^2}}) \geq 1 - 7\kappa^k,
\end{equation*}
which proves the theorem if we consider $\kappa \to 0$, parameter that uniformly controls the convergence of $\sigma$ towards $0$.

\subsection{Proof of the intermediate results}
\label{ss:prooflemmas}

Several propositions are proven here.

\subsubsection{Proof of Proposition~\ref{prop:Tkappa:uniform}}
\label{proof:Tkappa:uniform}

\medskip

In the following, we remind the reader that we do not emphasize the dependence on $\sigma$, but it will appear everywhere in the computations.  

1. As was mentioned above, the invariant probability measure of self-interacting diffusion (and, at the same time, the weak-* limit of its empirical measure a.s.) is the unique solution to the equation
\begin{equation*}
    \mu_\infty = \Pi_\sigma(\mu_\infty),
\end{equation*}
where $\Pi_\sigma$ is defined as:
\[
\Pi_\sigma(\mu)(x) =  e^{-2(V+W*\mu)(x) /\sigma^2} / \int e^{-2(V+W*\mu)(z) /\sigma^2} \rmd z\,.
\]

The same invariant probability measure appears in the self-stabilizing diffusion, small-noise limit of which was studied in \cite{HT-EJP}. There, authors studied the case of double-wells potentials which is more general then our diffusion. In this paper the result, that can be transformed in our context as following, was proved. If the moments of invariant probability measures $\mu_\infty$ are uniformly bounded with respect to $\sigma$, then $\delta_m$ is the weak-* limit of $\mu_\infty$ with $\sigma \to 0$ a.s. Note, that indeed, moments of $\mu_t$ are uniformly bounded for any $t>0$. Indeed, this is due to the fact that $\mu_t \in K_{\alpha, C}$ for any $t > 0$ and for some $\alpha, C$ that do not depend on $\sigma$ (Proposition~\ref{prop:mu_t_K_alpha}). It proves that
\begin{equation*}
    \mu_\infty \xrightarrow[\sigma \to 0]{\text{weak-*}} \delta_m \text{  a.s.}
\end{equation*}

2. Let us consider the expectation $\E \left(\mathbb{W}_{2k}\left(\mu_t; \delta_m\right)\right)$. First, let us show its existence. To do that, we use the fact that for any $t>0$, $\mu_t \in K_{\alpha, C}$ almost surely and get
\begin{equation*}
    \mathbb{W}_{2k}\left(\mu_t; \delta_m\right) \leq \left( 2^{2k - 1} \int |x|^{2k} \mu_t(\rmd x) + 2^{2k - 1} |m|^{2k}\right)^{1/(2k)} \leq \text{Const},
\end{equation*}
where the last constant depends only on $\alpha, C, m$ and $k$. Therefore, since the random variable is bounded by a constant almost surely, expectation exists.

3. Now, we can finish the proof by separating the expectation of the distance between $\mu_t$ and $\delta_m$ into two parts and find the limit: 
\begin{equation*}
    \E\left(\mathbb{W}_{2k}\left(\mu_t; \delta_m\right)\right) \leq \E\left(\mathbb{W}_{2k}\left(\mu_t; \mu_\infty\right)\right) + \E\left(\mathbb{W}_{2k}\left(\mu_\infty; \delta_m \right)\right) \xrightarrow[\substack{t \to \infty \\ \sigma \to 0}]{} 0,
\end{equation*}
where the limit is not just iterated, but holds for the pair $(t, \sigma)$, since the rate of convergence of $\mu_t$ towards $\mu_{\infty}$ in time does not depend on $\sigma$, which was shown in Proposition \ref{thm:KK12}. Therefore, for any $\kappa > 0$ we can find $\sigma_0$ small enough and~$t_0$ big enough such that $T_\kappa(\sigma) < T_\kappa < \infty$ for any $\sigma < \sigma_0$, which does not only prove existence and finiteness of $T_\kappa(\sigma)$, but also its uniformness with respect to~$\sigma$.

\subsubsection{Proof of Proposition~\ref{prop:apple}}
\label{proof:apple}

\medskip

First of all, we fix some $\xi$ and introduce the following stopping time $\mathcal{T} := \inf\{t: |X_t^\sigma - \psi_t|^2 \geq \xi\}$. We apply It\^o formula and get the following result, for  $\omega \in \{\mathcal{T} > t\}$ (the choice of this event will be clear further) : 
\begin{equation*}\label{eq:X_t-psi_t}
\begin{aligned}
    |X_t - \psi_t|^2 &= 2\int_{0}^t (X_s - \psi_s, \rmd{X_s}-\rmd{\psi_s}) + d \sigma^2 t \\
    & \leq d\sigma^2t -2\int_0^t (X_s - \psi_s; \nabla V(X_s) - \nabla V(\psi_s)) \rmd s \\
    & \quad - \int_0^t\frac{2}{s} \int_0^s (X_s - \psi_s, \nabla W(X_s - X_z) - \nabla W(\psi_s - \psi_z))\rmd z\rmd s\\
    & \quad + 2\sigma \int_0^t (X_s - \psi_s , \rmd B_s).
 \end{aligned}
\end{equation*}   
Let $\text{Lip}_{\nabla W}^{K^\prime}$ be a Lipschitz constant of $\nabla W$ inside the following compact
\[
K^\prime:= \{x: |x - \psi_t|^2 \leq \xi \text{, for some } t > 0\}\,.
\]

Due to our assumptions, this set is indeed compact at least for small $\xi$, which we can decrease without any loss of generality. We remind that $\rho$ is the convexity constant of $V$. We thus have 
\begin{eqnarray}\label{eq:X_t-psi_tbis}
    |X_t - \psi_t|^2    
    & \leq & d\sigma^2 t - 2\rho\int_0^t |X_s - \psi_s|^2\rmd s + 2\sigma \int_0^t(X_s - \psi_s , \rmd B_s)\notag \\
    &+& \text{Lip}_{\nabla W}^{K^\prime}\int_0^t\frac{2}{s}\int_0^s \big(|X_s - \psi_s|^2 + |X_s - \psi_s|\cdot|X_z - \psi_z|\big) \rmd z\notag \\
    & \leq & d\sigma^2 t - 2\rho\int_0^t |X_s - \psi_s|^2\rmd s + 2\sigma \int_0^t(X_s - \psi_s , \rmd B_s) \\
    &+ & \text{Lip}_{\nabla W}^{K^\prime}\int_0^t\frac{1}{s}\int_0^s \big(3|X_s - \psi_s|^2 + |X_z - \psi_z|^2\big) \rmd z.\notag
\end{eqnarray}
Note then that by the Burkholder-Davis-Gundy inequality, we get for some constant $C > 0$:
\begin{equation*}
\begin{aligned}
    \E \left(\sup_{[0,t \wedge \mathcal{T}]} \left|2\sigma \int_0^s (X_z-\psi_z, \rmd B_z )\right| \right) &\le C \sigma^2 \E \sqrt{\int_0^{t\wedge \mathcal{T}} |X_s - \psi_s|^2 \rmd s}\\
    & \le C \sigma^2 \sqrt{\int_0^t \E \left(\sup_{z \in [0, s \wedge \mathcal{T}]}(|X_z - \psi_z|^2)\right) \rmd s}. 
\end{aligned}
\end{equation*}

Let us consider the following random variable: $\sup_{s \in [0; t \wedge \mathcal{T}]} |X_s - \psi_s|^2$. The fact that we consider the supremum before time $t \wedge \mathcal{T}$ gives us that for any~$\omega$ we consider only such $s$, that $s \leq \mathcal{T}(\omega)$, which in turn means that we can apply estimation \eqref{eq:X_t-psi_tbis} for any $s \in [0, t\wedge \mathcal{T}]$. We also remind that $t \leq T$ and derive:
\begin{equation*}
\begin{aligned}
    \E \left(\sup_{s \in [0; t \wedge \mathcal{T}]} |X_s - \psi_s|^2\right) &\leq d\sigma^2T+ C \sigma^2 \sqrt{\int_0^t \E \left(\sup_{z \in [0 , s \wedge \mathcal{T}]} (|X_z - \psi_z|^2)\right) \rmd z} \\
    &\quad + 4 \text{Lip}_{\nabla W}^{K^\prime} \int_0^t \E \left(\sup_{z \in [0, s \wedge \mathcal{T}]} |X_z - \psi_z|^2\right) \rmd s \\
    &\leq  d\sigma^2T + \frac{C \sigma^2}{2} \left[1 +  T\E\left( \sup_{s \in [0 , t \wedge \mathcal{T}]} (|X_s - \psi_s|^2)\right)  \right] \\
    & \quad + 4 \text{Lip}_{\nabla W}^{K^\prime} \int_0^t \E \left(\sup_{z \in [0, s \wedge \mathcal{T}]} |X_z - \psi_z|^2\right) \rmd s,
\end{aligned}
\end{equation*}
where in the last inequality we used $\sqrt{x} \leq (1 + x)/2$. Now, if we denote\\ $u_t := \E \left(\sup_{s \in [0; t \wedge \mathcal{T}]} |X_s - \psi_s|^2\right)$, we have
\begin{equation*}
    u_t \leq \frac{1}{1 - C T \sigma^2/2} \left(\frac{2Td + C}{2}\sigma^2 + 4 \text{Lip}_{\nabla W}^{K^\prime} \int_0^t u_s \rmd s \right),
\end{equation*}
for small enough $\sigma$ (such that $1 - C T \sigma^2/2 > 0$). Thus, using Gr\"onwall lemma, we get
\begin{equation}\label{eq:ut_bound}
    u_t \leq \frac{(2Td + C) \sigma^2}{2(1 - C T \sigma^2/2)}\exp\left\{\frac{4 \text{Lip}_{\nabla W}^{K^\prime} }{1 - C T \sigma^2/2} T\right\} = O(\sigma^2).
\end{equation}
This in particular means, that $\E\left( \sup_{s \in [0; T \wedge \mathcal{T}]} |X_s - \psi_s|^2 \right) \leq O(\sigma^2)$. Nevertheless, to show the necessary result, we have to get rid of the stopping time $\mathcal{T}$ in the previous equation. It is sufficient to show, that $\mathbb{P} (\mathcal{T} \leq T) \xrightarrow[\sigma \to 0]{} 0$. 

Indeed, by its definition, $\mathcal{T}$ is the first time when the difference $|X_t - \psi_t|^2$ reaches $\xi$. But under the assumption $\mathcal{T} \leq T$ and due to \eqref{eq:ut_bound}, by decreasing $\sigma$ we can control $|X_t - \psi_t|^2$ and make it small enough, such that $|X_{\mathcal{T}} - \psi_{\mathcal{T}}|^2 < \xi$ (in some sense), which contradicts the definition of $\mathcal{T}$. Rigorously, 
\begin{equation*}
    \mathcal{T} < T \Rightarrow \sup_{[0, \mathcal{T}\wedge T]} |X_s - \psi_s|^2 = \sup_{[0, \mathcal{T}]} |X_s - \psi_s|^2 \ge \xi.
\end{equation*}
Thereby, 
\begin{equation*}
    \mathbb{P} (\mathcal{T} < T) \leq \mathbb{P}(\sup_{[0, \mathcal{T}\wedge T]} |X_s - \psi_s|^2 \ge \xi) \leq O(\sigma^2),
\end{equation*}
by Markov inequality.

To conclude the proof of Proposition~\ref{prop:apple}, we consider
\begin{equation*}
\begin{aligned}
    \mathbb{P}\left(\sup_{t\in [0;T]}|X_t-\psi_t(x_0)|^2>\xi\right) & \leq \mathbb{P}\left(\sup_{t\in [0;T]}|X_t-\psi_t(x_0)|^2>\xi, \mathcal{T} > T\right) \\
    & \quad\quad + \mathbb{P} \left( \mathcal{T} \leq T \right) \\
    & \leq \mathbb{P}\left(\sup_{t\in [0;T \wedge \mathcal{T}]}|X_t-\psi_t(x_0)|^2>\xi \right) + O(\sigma^2) \\
    & \leq O(\sigma^2),
\end{aligned}
\end{equation*}
by Markov inequality and \eqref{eq:ut_bound}, which completes the proof.

\subsubsection{Proof of Proposition~\ref{prop:coupling}}
\label{proof:coupling}

Let us define 
\[
W_m(x):=V(x)+ W(x-m)\quad\mbox{and}\quad W_{\mu_t}(x):=V(x)+W\ast\mu_t(x)\,,
\]
with the occupation measure $\mu_t:=\frac{1}{t}\int_0^t\delta_{X_s}\rmd s$. For any~$t\geq T_\kappa(\sigma)$, we have

\begin{equation*}
\rmd\left|X_t-Y_t\right|^2=-2\left(X_t-Y_t\,;\,\nabla W_{\mu_t}(X_t)-\nabla W_m(Y_t)\right)\rmd t.
\end{equation*}
We thus have
\begin{align*}
\frac{\rmd}{\rmd t}\left|X_t-Y_t\right|^2=&-2\left( X_t-Y_t\,;\,\nabla W_{\mu_t}\left(X_t\right)-\nabla W_{\mu_t}\left(Y_t\right)\right)\\
&+2\left( X_t-Y_t\,;\,\nabla W(Y_t-m)-\nabla W\ast\mu_t(Y_t)\right)\,.
\end{align*}
However, $\nabla^2W_{\mu_t}=\nabla^2V+\nabla^2 W\ast\mu_t\geq(\rho+\alpha){\rm Id}$ with $\rho+\alpha>0$. So, putting $\gamma(t):=\left|X_t-Y_t\right|^2$, Cauchy-Schwarz inequality yields to
\begin{equation*}
\gamma'(t)\leq-2\left(\alpha+\rho\right)\gamma(t)+2\sqrt{\gamma(t)}\left|\nabla W(Y_t-m)-\nabla W\ast\mu_t(Y_t)\right|\,.
\end{equation*}
However, by the growth condition \eqref{domination} on $W$, we have for any probability measures $\mu,\nu$ the following control $$\left|\nabla W\ast\mu(x)-\nabla W\ast\nu(x)\right|\leq C\left(1+|x|^{2k}\right)\mathbb{W}_{2k}^{2k}\left(\mu,\nu\right)$$ where $2k$ is the degree of the polynomial $P$ introduced in Assumption~\ref{assu:growth1}. We introduce the set 
\begin{equation*}
\mathcal{A}_\kappa:=\left\{\omega\in\Omega\,\,:\,\,\mathbb{W}_{2k}^{2k}\left(\mu_t,\delta_m\right)\leq\kappa^k\right\}\,.
\end{equation*}

By Markov inequality, we have $\PP\left(\mathcal{A}_\kappa^c\right)\leq7\kappa^k$ then $\mathbb{P}\left(\mathcal{A}_\kappa\right)\geq1-\kappa^k$. This implies for any $t\geq T_\kappa(\sigma)$ and for any $\omega\in\mathcal{A}_\kappa$:
\begin{equation*}
\gamma'(t)\leq-2\left(\alpha+\rho\right)\gamma(t)+2C\kappa^k\sqrt{\gamma(t)}\left(1+\left|Y_t\right|^{2k}\right)\,.
\end{equation*}
 However, $\gamma(t)=0$ for any $t\leq T_\kappa(\sigma)$. This means that
$$\left\{t\geq0\,\,:\,\,\gamma(t)>\frac{C^2\kappa^{2k}\left(1+\left|Y_t\right|^{2k}\right)^2}{(\alpha+\rho)^2}\right\}\subset\Big\{t\geq0\,\,:\,\,\gamma'(t)<0\Big\}\,.$$

By \cite[Lemma~3.7]{BRTV}, we deduce that
\begin{equation*}
\sup_{T_\kappa(\sigma)\leq t\leq\exp\left[\frac{2H+10}{\sigma^2}\right]}\gamma(t)\leq\frac{C^2\kappa^{2k}\left(1+\sup_{T_\kappa(\sigma)\leq t\leq\exp\left[\frac{2H+10}{\sigma^2}\right]}\left|Y_t\right|^{2k}\right)^2}{(\alpha+\rho)^2}\,,
\end{equation*}
if $\omega\in\mathcal{A}_\kappa$. We now consider $R>0$ such that the exit-cost of the diffusion~$Y$ from the ball of center $m$ and radius $R$ is at least $H+6$, meaning that
\[
\inf \{V(x) + W(x~\!-~\!m) - V(m): x \in B(m, R)\} \geq H + 6\,.
\]

Then, by Freidlin-Wentzell theory, we deduce that
\[
\lim_{\sigma\to 0}\mathbb{P}\left(\sup_{T_\kappa(\sigma)\leq t\leq\exp\left[\frac{2H+10}{\sigma^2}\right]}\left|Y_t - m\right|\geq R\right)=0\,.
\]
However, we have
\begin{align*}
&\mathbb{P}\left(\sup_{T_\kappa(\sigma)\leq t\leq\exp\left[\frac{2H+10}{\sigma^2}\right]} \left|X_t-Y_t\right|\geq\xi\right)\\ 
&\quad \leq \mathbb{P}\left(\sup_{T_\kappa(\sigma)\leq t\leq\exp\left[\frac{2H+10}{\sigma^2}\right]}\left|Y_t - m\right|\geq R\right)\\
&\quad + \mathbb{P}\left(\sup_{T_\kappa(\sigma)\leq t\leq\exp\left[\frac{2H+10}{\sigma^2}\right]}\gamma(t)\geq\xi^2,\sup_{T_\kappa(\sigma)\leq t\leq\exp\left[\frac{2H+10}{\sigma^2}\right]}\left|Y_t - m\right|<R\right)\\
&\quad+ \mathbb{P}\left(\mathcal{A}_\kappa^c\right)\,.
\end{align*}
The first term tends to $0$ as $\sigma$ goes to $0$. The second term is equal to $0$ provided that
\[
\xi>\frac{C\kappa^{k} \big(1+ 2^{2k - 1}(R + |m|^{2k}) \big)}{\alpha+\rho}\,.
\]

In other words, if $\kappa$ is small enough, the second term is equal to $0$ uniformly with respect to $\sigma$. The third term is less than $\sqrt{\kappa}$. This concludes the proof.

\subsection{Proof for the exit-location result}
\label{ss:proofel}

We can apply Theorem~\ref{dell} to the level sets of the potential $W_m:=V+W\ast\delta_m$.

By definition of $\mathcal{N}$ in Corollary~\ref{hphp}, there exists a constant $\xi>0$ such that
\[
\inf_{z\in\mathcal{N}}\left(V(z)+W(z-m)-V(m)\right)=H+3\xi\,.
\]

We introduce the set 
\begin{eqnarray*}
\mathcal{K}_{H+2\xi}:=\left\{x\in\mathbb{R}^d\,\,:\,\,V(x)+W(x-m)-V(m)<H+2\xi\right\}\,.
\end{eqnarray*}
If we denote by $\tau_\xi$ the first exit-time of $X$ from $\mathcal{K}_{H+2\xi}$, then we obtain
\begin{eqnarray}
\label{oz}
\lim_{\sigma \to0}\mathbb{P}\left\{\exp\left[\frac{2}{\sigma^2}\left(H+2\xi-\eta\right)\right]<\tau_\xi<\exp\left[\frac{2}{\sigma^2}\left(H+2\xi+\eta\right)\right]\right\}=1\,,
\end{eqnarray}
for any $\eta>0$. By construction of $\mathcal{K}_{H+2\xi}$, $\mathcal{N}\subset\mathcal{K}_{H+2\xi}^c$, which implies
\begin{align*}
\mathbb{P}\left\{X_{\tau}\in\mathcal{N}\right\}\leq&\mathbb{P}\left\{X_{\tau}\notin\mathcal{K}_{H+2\xi}\right\}\\
\leq&\mathbb{P}\left\{\tau_{\xi}\leq\tau\right\}\\
\leq&\mathbb{P}\left\{\tau_{\xi}\leq\exp\left[\frac{2(H+3\xi)}{\sigma^2}\right]\right\}+\mathbb{P}\left\{\exp\left[\frac{2H+\xi}{\sigma^2}\right]\leq\tau\right\}\,.
\end{align*}
Applying \eqref{oz} with $\eta:=\xi$ to the first term and Theorem \ref{dell} to the second one, we obtain the result.

\appendix
\section{Discussions on extension}
\label{ss:extension}

In this Section, we provide some ideas on how the results of the paper can be extended.

\medskip

First, we can modify our equation by adding a reflection at the boundary as it was done for McKean-Vlasov case, for example in \cite{ADRRST}. Note, that if the boundary of reflection contains the closure of the domain from which we  want to exit, the result for exit-time does not change and is immediate, since, unlike in McKean-Vlasov case, there is no interaction with the law of the process for self-interacting diffusion.

\medskip

Second, in this paper we take a diffusion coefficient which is proportional to the identity matrix. However, it could be relevant for some problems related to optimization to consider a more general diffusion coefficient.

\medskip

Third, in the current work we do not derive an Arrhenius law, that is to say the convergence of $\frac{\sigma^2}{2}\log(\EE[\tau])$ towards $H$. To obtain such a result, it requires to use the large deviations techniques instead of the coupling method that is used here. 

\medskip

Fourth, we point out that the potentials $V$ and $W$ are both assumed to be uniformly convex. The techniques used in this paper are not adapted for a more general case. One shall use the techniques close to \cite{DZ,FW98} to relax the convexity assumptions. 

\medskip

Finally, we could also study SDEs where the nonlinear part of the drift is more general. that is to say:
\begin{equation}
\label{eq:sid:ultrageneral}
X_t=x_0+\sigma B_t-\int_0^t\nabla V(X_s)ds-\int_0^tb(X_s,\mu_s)ds\,,
\end{equation}
where $\mu_s:=\frac{1}{t}\int_0^t\delta_{X_s}ds$ and $x\mapsto b(x,\mu)$ is differentiable whereas $\mu\mapsto b(x,\mu)$ is $L$-differentiable, see~\cite{trucdelions} and references therein. However, this will require some adaptations of our methods.

\medskip

{\bf Acknowledgments: }\emph{This work is supported by the French ANR grant METANOLIN (ANR-19-CE40-0009).}

\bibliographystyle{alpha}
\bibliography{ADMKT}

\end{document}